\documentclass[10pt]{article}
\usepackage{amsmath, amssymb, amsthm, amscd}
\numberwithin{equation}{section}

\def\p{\partial}

\def\cH{{\cal H}}

\def\C{{\mathbb C}}

\def\C{{\mathfrak C}}

\def\cH{{\mathcal H}}

\newtheorem{prop}{Proposition}[section]
\newtheorem{theo}[prop]{Theorem}
\newtheorem{lemma}[prop]{Lemma}
\newtheorem{cor}[prop]{Corollary}
\newtheorem{rem}[prop]{Remark}

\newtheorem{defi}[prop]{Definition}

\let\lra=\longrightarrow

\def\mapright\#1{\,\smash{\mathop{\lra}\limits^{\#1}}\,}

\begin{document}
\bibliographystyle{plain}
\title{The space of volume forms}
\author{Xiuxiong Chen;
\; Weiyong He}

\date{}
\maketitle 
\abstract{Donaldson introduced an interesting geometric structure (the Donaldson metric) on the space of volume forms for any compact Riemannian manifold, which has nonpositive sectional curvature formally. The geodesic equation and its perturbed equations are  fully nonlinear elliptic equations .  These equations are also equivalent to two free boundary problems of the Laplacian equation and it also has relationship with many interesting problems, such as Nahm's equation. In this paper we solve these equations and demonstrate the geometric structure of the space of volume forms; in particular, we show that the space of volume forms with the Donaldson metric is a metric space with non-positive curvature in Alexanderov sense.}

\tableofcontents

\section{Introduction}

On any compact K\"ahler manifold, Mabuchi \cite{Mabuchi87}, Semmes
\cite{Semmes96}, and Donaldson \cite{Donaldson1997} introduced a
Weil-Peterson type metric in the space of K\"ahler metrics and proved
 that it is a formally non-positively curved symmetric space of
``non-compact" type. According to \cite{Semmes96}, the geodesic
equation can be transformed into a homogenous complex Monge-Ampere
equation. In \cite{Donaldson1997}, Donaldson proposed an ambitious
program relating the geometry of this infinite dimensional space
to the core problems in K\"ahler geometry, such as the uniqueness
and the existence problems for constant scalar curvature K\"ahler
metrics and  its relation to the stability of the underlying
polarization. In \cite{Chen2000}, the first author  solved the geodesic equation in $C^{1, 1}$ sense with intriguing
applications in K\"ahler geometry. This provides somewhat
technical foundation to this ambitious program of Donaldson.  Exciting
progress is achieved in the last few years in that subject and the readers are
encouraged to read \cite{Donaldson2002, Arezzotian2003, Donaldson2005, Chen2005}...  for further
references in that subject.\\

Is there a corresponding story on Riemannian side?  After all,  in
complex dimension 1, complex and real geometry coincide.  Perhaps
with a bit more imagination, this ``dual" or ``companion version"
of Donaldson's original program might also have important
applications in Riemannian geometry.   In a recent paper
\cite{Donaldson2007}, Donaldson tells a potentially exciting story
in the space of volume forms, more or less, mirror to his program
in K\"ahler geometry. It of courses comes with  twists of new
ideas. He pointed out that the existence problem for geodesic
segment  is related to  a few renowned problems in PDE such as
regularity for some free boundary problems, Nahm's equation etc. From
PDE point of view, the geodesic equation is similar to its
analogues equation in K\"ahler setting, but has difference in a significant way. \\

Let us be more specific. In \cite{Donaldson2007}, Donaldson
introduced a Weil-Peterson type metric on the space of volume
forms (normalized) on any Riemannian manifold $(X, g)$ with fixed
total volume. This infinite dimension space can be parameterized
by smooth functions such that
\[
{\cal H} = \{\phi\in  C^\infty (X): 1 + \triangle_g \phi > 0\}.
\]
This is a locally Euclidean space.  The tangent space is exactly
$C^\infty(X)$, up to addition of a constant. For any $\delta \phi\in T_\phi\cH$, 
the metric is given by
\[
\|\delta \phi\|_\phi^2 =  \int_X\; (\delta \phi)^2 (1 +
\triangle_g \phi) dg.
\]
The energy function on a path $\Phi(t): [0, 1]\rightarrow \cH$ is
defined as
\[
E(\Phi(t)=\int_0^1\int_X|\dot \Phi|^2(1+\triangle \Phi)dg.
\]
Then the geodesic equation is
\begin{equation}\label{E-1-1}
\Phi_{tt}(1+\triangle \Phi) -|\nabla \Phi_t|^2_g = 0.
\end{equation}
This is a fully nonlinear degenerated elliptic equation. To approach this
equation, Donaldson introduced a perturbed equation of the geodesic
equation
\begin{equation}\label{E-1-2}
\Phi_{tt}(1+\triangle \Phi) -|\nabla \Phi_t|^2_g =\epsilon,
\end{equation}
for any $\epsilon>0.$ The equation (\ref{E-1-2}) can be also
formulated  as the other two equivalent free boundary problems
according to \cite{Donaldson2007}.\\

As usual, the geodesic equation (\ref{E-1-1}) tells exactly how to
define the Levi-Civita connection in $\cal H.\;$ Donaldson showed
that   $\cal H$ is formally a non-positively curved  space. Donaldson asked
if there exists a smooth geodesic segment  between any two points in $\cH$. In this paper, we give a partial
answer to this question.
\begin{theo}\label{T-1-1}For any two points $\phi_0, \phi_1 \in \cH$ and any $\epsilon>0$, there
exists a smooth solution of (\ref{E-1-2}), $\Phi(t): [0,
1]\rightarrow \cH$ which connects $\phi_0, \phi_1$.
\end{theo}
In particular, we can prove that the (weak) $C^2$ bound of the
solution is independent of $\inf\epsilon.$ The weak $C^2$ means $\triangle \Phi, \Phi_{tt}, \nabla \Phi_t$ are bounded, while $\nabla^2\Phi$ might not. Hence,

\begin{theo}\label{T-1-2} For any two points $\phi_0,
\phi_0\in \cH$, there exists a weakly $C^{2}$ geodesic segment
$\Phi: [0, 1]\rightarrow \overline{\cH}$  which connects $\phi_0,
\phi_1 \in \cal H,$  where $\overline{\cH}$ is the closure of
$\cH$ under the weak $C^2$ topology. \end{theo}

Following  \cite{CC2002}, we
can also prove that
\begin{theo}\label{T-1-3} The infinite dimensional space $\cal H$ is a
non-positively curved metric space in the sense of Alexandrov.
\end{theo}

According to Donaldson, the geodesic equation (\ref{E-1-1}) and
the equation (\ref{E-1-2}) are relevant to many other interesting
problems, especially the Nahm's equation. 
Also the space of volume forms is of fundamental interest in many subjects, such as optimal transportation theory.
Theorems  \ref{T-1-1}, \ref{T-1-2} might have some applications in these related subjects.\\

We derive the a priori estimates and use the method of continuity to solve \eqref{E-1-1} and \eqref{E-1-2}. 
For any function $\Phi\in C^2(M\times [0, 1])$, define
\[
Q(D^2\Phi)=\Phi_{tt}(1+\triangle \Phi)-|\nabla \Phi_t|^2.
\]
One can derive 
the a priori estimates for 
\[
Q(D^2\Phi)=f>0.
\]
The estimates can actually  be done; for example, $|\Phi|_{C^1}$ will depend on $|\nabla f^{1/2}|$ (c.f \cite{He09}). 
One can also choose paths as follows. Set
\begin{eqnarray}
P(s, D^2 \Phi)=s\;Q(D^2 \Phi)+(1-s) (\Phi_{tt}+\triangle \Phi).
\end{eqnarray}
We want to solve the following equation for any $s\in [0, 1]$ and
$\epsilon>0$
\begin{equation}\label{E-1-3}
P \left(D^2 \Phi(\cdot, t,s,\epsilon)\right) =
\epsilon,\end{equation}  with the boundary condition
\begin{equation*} \Phi(\cdot, 0, s, \epsilon) = \phi_0, \Phi(\cdot,
1, s, \epsilon) = \phi_1.
\end{equation*}
These kind of paths are also used often to deal with fully nonlinear equations; for example, see \cite{Giltru1998} and \cite{Caca}. 
In our case the right hand side of \eqref{E-1-3} becomes a constant and the dependence  on $|\nabla f^{1/2}|$ is automatically gone. Otherwise two cases are similar ($P$ is just a slight modification of $Q$ as an operator and has the similar properties, see \eqref{concave} for example); for simplicity we shall still use \eqref{E-1-3}. 

When $s=0,$ \eqref{E-1-3} is a standard Laplacian equation and the
Dirichlet problem is always solvable  with respect to any
$\epsilon >  0\;$ and any smooth $\phi_0, \phi_1 \in \cal H.\;$
For simplicity of notations, we always assume that $\epsilon\leq
1$. Otherwise, the estimates depend on $\epsilon$ when $\epsilon$
is large.  The main result  in this paper
is the following a priori estimates
\begin{theo}\label{T-1-4} For any smooth solution to the equation (1.4), there is a uniform bound
on $|\Phi|_{C^0}$, $|\Phi|_{C^1}$, $\triangle \Phi, \Phi_{tt}$ and
$ |\nabla \Phi_{t}|$, independent of both $\epsilon$ and $s$.
\end{theo}

{\bf Organization:} In Section 2 we summarize Donaldson's theory
on the space of volume forms. In Section 3 we derive the {\it a
priori} estimates to solve the equations.  We assume first
that the Ricci curvature of the background metric is non-negative.
With this assumption the computation is straightforward. In Section 4 we derive the a priori
$C^2$ estimates without Ricci assumption. In Section 5 we discuss
the geometric structure of $\cH$. In particular we prove that
$\cH$ is a non-positively curved metric space in the sense of
Alexandrov.\\

\noindent{\bf Acknowledgement:} Both authors would like to thank Y. Yuan for
pointing out an oversight when applying Evans-Krylov theory  in a previous version of the paper.  Both authors would also like to thank the referees for useful suggestions and comments, and for spotting a gap in the argument of Theorem \ref{T-7-8}. 
WYH would like to thank Professor N. Ghoussoub
for valuable discussions for the space of volume forms.
XXC is partially supported by NSF and WYH was partially supported by a PIMS postdoc fellowship at University of British Columbia. 

\section{Geometric structure of $\cH$}
In this section we briefly summarize Donaldson's theory
\cite{Donaldson2007} on the space of volume forms. The readers are
encouraged to  refer \cite{Donaldson1997, Chen2000} for
more details. Given a compact Riemannian manifold $(X, g)$, let
$\cH$ be the set of smooth functions $\phi$ on $X$ such that
$1+\triangle \phi>0$. (We use the sign convention that $\triangle$
is a negative operator, which is opposite to the sign convention
in \cite{Donaldson2007}.) We now introduce a $L^2$ metric in this
space. Clearly, the tangent space $T\cH$ is $C^{\infty}(X)$.
Define the norm of tangent vector $\delta \phi$ at a point $\phi$
by \[\|\delta \phi\|^2_{\phi}=\int_X(\delta\phi)^2(1+\triangle
\phi)dg.
\]
Thus a path $\phi(t)$ in $\cH$, parameterised by $t\in [0, 1]$ say,
is simply a function on $X\times [0, 1]$ and the ``energy" of the
path is
\begin{equation}\label{E-2-1}
E(\phi(t))=\frac{1}{2}\int_0^1\int_X\left(\frac{\p \phi}{\p
t}\right)^2(1+\triangle \phi)dgdt.
\end{equation}

It is straightforward to write down the Euler-Lagrange equations
associated to the energy (\ref{E-2-1}). These are
\begin{equation*}
\ddot{\phi}=\frac{|\nabla \dot\phi|^2}{1+\triangle \phi}.
\end{equation*}
These equations define the geodesics in $\cH$. We can read off the
Levi-Civita connection of the metric from this geodesic equation,
as follows. Let $\phi(t)$ be any path in $\cH$ and $\psi(t)$ be
another function $X\times [0, 1]$, which we regard as a vector
field along the path $\phi(t)$. Then the covariant derivative of
$\psi$ along the path is given by
\begin{equation*}
D_t\psi=\frac{d\psi}{dt}+(W_t, \nabla_X\psi),
\end{equation*}
where
\[
W_t=\frac{-1}{1+\triangle \phi}\nabla_X\dot\phi.
\]
This has an important consequence for the {\it holonomy group} of
the manifold $\cH$. It is shown \cite{Donaldson2007} that the
holonomy group of $\cH$ is contained in the group of
volume-preserving diffeomorphisms of $(X, d\mu_0),$ where
$d\mu_0=(1+\triangle\phi_0)dg.$ Also Donaldson proved that the
sectional curvature of the manifold $\cH$ is formally non-positive. Let
$\phi$ be a point of $\cH$ and let $\alpha, \beta$ be tangent
vectors to $\cH$ at $\phi$, so $\alpha, \beta$ are just functions
on $X$. The curvature $R_{\alpha, \beta}$ is a linear map from
tangent vectors to tangent vectors. Donaldson showed the following

\vspace{2mm} {\bf Theorem A.} {\it The curvature of $\cH$ is given
by
\[
R_{\alpha, \beta}(\psi)=(\nu_{\alpha, \beta}, \nabla \psi),
\]
where the vector field
\[
\nu_{\alpha, \beta}=\frac{1}{1+\triangle
\phi}\mbox{curl}\left(\frac{1}{1+\triangle \phi}\nabla
\alpha\times \nabla \beta\right).
\]
The sectional curvature is defined by
\[
K_{\alpha, \beta}=\langle R_{\alpha, \beta}(\alpha), \beta\rangle.
\]
In particular  \[ K_{\alpha, \beta}=-\int_X\frac{1}{1+\triangle
\phi}|d\alpha\wedge d\beta|^2dg\leq 0.\]}

We define a functional on $\cH$ for paths by
\[
V(\phi)=\int_X\phi dg.
\]
This function is convex along geodesics in $\cH$, since the
geodesic equation implies that $\ddot\phi\geq 0$. Now introduce a
real parameter $\epsilon\geq 0$ and consider the functional on
paths in $\cH$:
\begin{equation*}
E_{\epsilon}(\phi(t))=E(\phi(t))+\epsilon\int_0^1 V(\phi)dt,
\end{equation*}
corresponding to the motion of a particle in the potential
$-\epsilon V$. The Euler-Lagrange equations are
\begin{equation*}
\ddot\phi=\frac{|\nabla_X\dot \phi|^2+\epsilon}{1+\triangle \phi}.
\end{equation*}
This equation is equivalent to the other two free boundary
problems. For the detailed discussion of this equation and its
relation to free boundary problems and many other interesting
problems, we refer the readers to Donaldson \cite{Donaldson2007}.

\section{ Existence of the solution}
In this section we derive the {\it a priori} estimates to solve
equation (1.1).  For the easiness of presentation and convenience of the first time readers,
we  assume the background metric has non-negative  Ricci curvature. In this case, the
calculation is more streamliner which explains the main idea. The general case will be deferred  to Section 4.
\begin{theo}\label{T-3-1} Assume that $(X, g)$ has non-negative Ricci curvature.
For any $\phi_0, \phi_1\in \cH$, the equation
\begin{equation}\label{E-3-1} P \left(D^2 \Phi(\cdot,
t,s,\epsilon)\right) = \epsilon,\end{equation}  with the boundary
condition
\begin{equation}\label{3-2} \Phi(\cdot, 0, s, \epsilon) = \phi_0, \Phi(\cdot,
1, s, \epsilon) = \phi_1,
\end{equation} has a unique smooth solution with $\Phi\in \cH$ at
$s=1$. Moreover, \[ |\Phi|_{C^0}, |\Phi|_{C^1}, |\triangle \Phi|,
|\Phi_{tt}|, |\nabla \Phi_t|
\] are uniformly bounded, independent of $s$ and $\epsilon$.
\end{theo}

\subsection{Concavity and the continuity method}
The foundation relies on a convexity of the nonlinear
operator $Q$ which is showed by Donaldson in \cite{Donaldson2007}.
It is also used to show the uniqueness of the equation
(\ref{E-1-2}) in \cite{Donaldson2007}. Let $A$ be a symmetric
$(n+1)\times (n+1)$ matrix with entries $A_{ij},~0\leq i, j\leq
n$. Define
\[
Q(A)=A_{00}\sum_{i=1}^nA_{ii}-\sum_{i=1}^nA_{i0}^2.
\]
Thus $Q$ is a quadratic function on the vector space of symmetric
$(n+1)\times (n+1)$ matrices.

\begin{lemma}\label{L-3-2}(Donaldson \cite{Donaldson2007}) 1. If $A>0$, then
$Q(A)>0$ and if $A\geq 0$, $Q(A)\geq 0$.\\

2. If $A, B$ are two matrices with $Q(A)=Q(B)>0,$ and if the
entries $A_{00}, B_{00}$ are positive then for any $s\in [0, 1]$,
\[
Q(sA+(1-s)B)\geq Q(A), Q(A-B)\leq 0.
\]
Moreover, the equality holds if and only if $A_{ii}=B_{ii},
A_{i0}=B_{i0}$.
\end{lemma}

Lemma \ref{L-3-2} is shown by Donaldson \cite{Donaldson2007} using
some Lorentz geometry. One can also prove this through elementary
calculus. This lemma  is equivalent to the
following concavity of $\log Q (D^2\Phi)$.

\begin{lemma}\label{L-3-3} Consider the function
\[
f(x, y, z_1, \cdots, z_n)=\log{\left(xy-\sum z_i^2\right)}.
\]
Then $f$ is concave when $x>0, y>0, xy-\sum z_i^2>0$.
\end{lemma}
\begin{proof}
We can restate Lemma \ref{L-3-2}  as follows (taking $s=1/2$). Let
$x>0, y>0$ and $a, b>0$, if
\[
xy-\sum z_i^2=ab-\sum c_i^2>0,
\]
then
\begin{equation}\label{E-3-3}
\frac{1}{4}\left((x+a)(y+b)-\sum (z_i+c_i)^2\right)\geq xy-\sum
z_i^2.
\end{equation}
Obviously $f$ is smooth when $x>0, y>0, xy-\sum z_i^2>0$. It
suffices to prove that
\[
\frac{1}{2}\left(f(x, y, z_i)+f(\tilde x, \tilde y, \tilde
z_i)\right)\leq f\left(\frac{x+\tilde x}{2}, \frac{y+\tilde y}{2},
\frac{z_i+\tilde z_i}{2}\right).
\]
Namely, we need to show that
\[
\log{\left(xy-\sum z_i^2\right)}+\log{\left(\tilde x\tilde y-\sum
\tilde z_i^2\right)}\leq 2\log{\left(\left(\frac{x+\tilde
x}{2}\right)\left(\frac{y+\tilde y}{2}\right)-\sum
\left(\frac{z_i+\tilde z_i}{2}\right)^2\right)}.
\]
We can rewrite the above as
\begin{equation*}
\sqrt{\left(xy-\sum z_i^2\right)\left(\tilde x\tilde y-\sum \tilde
z_i^2\right)}\leq \left(\left(\frac{x+\tilde
x}{2}\right)\left(\frac{y+\tilde y}{2}\right)-\sum
\left(\frac{z_i+\tilde z_i}{2}\right)^2\right).
\end{equation*}
Suppose that for some $\lambda >0$,
\[
\tilde x\tilde y-\sum\tilde z_i^2=\lambda^2(xy-\sum z_i^2).
\]
Let
\[
\tilde x=\lambda a, ~~\tilde y=\lambda b,~~ \tilde z_i =\lambda
c_i,
\]
we get that \[ ab-\sum c_i^2=xy-\sum z_i^2.\]

The above reads
\begin{equation}\label{E-3-4}
\lambda (xy-\sum z_i^2)\leq\frac{1}{4}\left((x+\lambda
a)(y+\lambda b)-\sum (z_i+\lambda c_i)^2 \right).
\end{equation}
The right hand side of (\ref{E-3-4}) is
\[
\frac{1}{4}(1+\lambda^2)\left(xy-\sum
z_i^2\right)+\frac{\lambda}{4}(xb+ya-2z_ic_i).
\]
Note when $\lambda=1$,  (\ref{E-3-4}) is exactly (\ref{E-3-3}). We
can get that from (\ref{E-3-3})
\[
2(xy-\sum z_i^2)\leq (xb+ya-2z_ic_i).
\]
Hence (\ref{E-3-4}) follows from (\ref{E-3-3}). Note the equality
holds if and only if $\lambda=1, x=a, y=b, z_i=c_i$, namely
$x=\tilde x, y=\tilde y, z_i=\tilde z_i$.
\end{proof}
Replacing $y$ by $\sum_i y_i$, the above argument shows that \[h(x, y_1, \cdots, y_n, z_1, \cdots, z_n)=\log\left(x(\sum_i y_i)-\sum_iz_i^2\right)\] is also a concave function for $x>0, \sum_iy_i>0$ and $x(\sum_iy_i)-\sum_iz_i^2>0$. It follows that $\log Q(D^2\Phi)$ is a concave functional on $D^2\Phi$.\\

When $s=0$, the equation (\ref{E-3-1}) reads
\[\Phi_{tt}+\triangle \Phi =\epsilon.\]
There is a unique smooth solution to this Laplace equation with
boundary condition (\ref{3-2}). It is a uniformly elliptic linear
equation and its linearization has zero kernel with zero boundary
data. Hence the equation (\ref{E-3-1}) can be solved uniquely for $s$
sufficiently close to 0. We can define
\begin{equation*}
s_0=\sup \left\{\tilde s: P(s, \Phi)=\epsilon~ \mbox{has a unique
smooth solution for}~ s\in [0, \tilde s)\right\}.
\end{equation*}
Note that $s_0$ is uniformly bounded away from zero. Actually we
have
\begin{prop}\label{P-3-2}
There exists a positive constant $\delta=\delta(X, \phi_0,
\phi_1)$ such that
\[P(D^2 \Phi)=\epsilon
\]
has a smooth unique solution $\Phi$ for any $s\in [0, \delta]$.
Moreover, $\left|\Phi\right|_{C^{k}}$ is uniformly bounded when
$s\in [0, \delta]$.\end{prop}
\begin{proof}
The proof is an inverse theorem for the operator $P$. Consider
\[
P: [0, 1]\times C^{k, \alpha}\rightarrow C^{k-2, \alpha}
\]
for any $k\geq 2,~\alpha\in [0, 1)$. We know there exists a unique
smooth function $\Phi$ such that $P(0, \Phi)=\epsilon.$ Since
$dP_{\Phi}$ is invertible at $(0, \Phi)$, the proposition follows
from the inverse function theorem in the Banach space.
\end{proof}
Without loss of generality, we will assume that $s\geq \delta>0$
for the continuity family (\ref{E-3-1}). We observe that the
linearity  $dP$ is elliptic for any $s\in [0, s_0)$, where
\begin{equation}\label{E-3-5}dP(h)=(s\Phi_{tt}+(1-s))\triangle
h+(s(1+\triangle\Phi)+(1-s))h_{tt}-2s\Phi_{tk}h_{tk}.
\end{equation}
\begin{prop}\label{P-3-3}For any $s\in [0, s_0)$, if $\Phi$ is the unique
solution of (\ref{E-1-2}), then
\[
\Phi_{tt}+1+\triangle \Phi>0.
\]
It follows that $s\Phi_{tt}+(1-s),  s(1+\triangle \Phi)+(1-s)$ are
both positive.  Moreover, $dP$ is an elliptic operator.
\end{prop}
\begin{proof}Let $s$ be the first value such that at some point $p=(x_0, t_0)\in X\times [0, 1]$
\[\Phi_{tt}+1+\triangle \Phi=0,
\]
where $\Phi$ is the solution of
\[
P(s, D^2\Phi)=\epsilon.
\]
It follows that\[Q(D^2\Phi)(x_0, t_0)=\Phi_{tt}(1+\triangle
\Phi)-\left|\nabla \Phi_t\right|^2\leq 0.
\]
It follows that
\[
P(s, D^2\Phi)(x_0, t_0)=sQ(D^2\Phi)+(1-s) (\Phi_{tt}+\triangle
\Phi)\leq 0.
\]
Contradiction. Since $\Phi_{tt}+1+\triangle \Phi>0$, it follows
that both $s\Phi_{tt}+(1-s), s(1+\triangle \Phi)+(1-s)$ are
positive. To show $dP$ is elliptic, we see that
\[
(s\Phi_{tt}+(1-s) )(s(1+\triangle
\Phi)+(1-s))-s^2\Phi_{tk}^2=s\epsilon+(1-s)^2.
\]
It follows that the quadratic form\[
(s\Phi_{tt}+(1-s))\sum_k\xi_k^2+(s(1+\triangle\Phi)+(1-s))\xi_0^2-2s\Phi_{tk}\xi_0\xi_k>0\]
for any nonzero vector $\xi=(\xi_0, \cdots, \xi_{n+1})$. In
particular, $dP$ is elliptic.
\end{proof}
\begin{rem}Note that we do not have $Q(D^2\Phi)>0$ along the continuous
family. At $s=0$, the solution $\Phi(x, t)$ is not necessarily in
$\cH$. But if we have a smooth solution at $s=1$, then $\Phi(x,
t)\in \cH$.
\end{rem}

We shall then derive the {\it a priori} estimates and prove Theorem
\ref{T-3-1} at the end of this section. We assume
$0<\epsilon \leq 1$ throughout the paper.
\subsection{$C^{0}$ estimate}
The $C^{0}$ estimate follows from Lemma \ref{L-3-2} and the
maximum principle.
\begin{prop}\label{P-4-2} Let $\Phi(x, t)$ be a solution of the
equation (\ref{E-3-1}) for $s\in [0, s_0)$. Denote $ \Psi_{a}(x,
t)=a(1-t)t+(1-t)\phi_0+t\phi_1$ where $a$ is a fixed constant.
Then $\Phi$ satisfies the following a priori $C^0$ estimate
\[
\Psi_{-a}\leq \Phi\leq \Psi_{a},
\]
when $a$ is sufficiently big.
\end{prop}
\begin{proof}Assume contrary; then there is some point $p=(x_0, t_0)$ such that
$\Phi(x_0, t_0)>\Psi_a(x_0, t_0)$ and so $\Phi-\Psi_a$ obtains its
maximum interior at $q$. At the point $q$, we have \[\Phi_{tt}\leq
{\Psi_{a}}_{,tt}=-2a, \triangle \Phi\leq \triangle
\Psi_{a}=(1-t)\triangle \phi_0+t\triangle \phi_1.\] However by
Proposition \ref{P-3-3}, \[\Phi_{tt}>-\frac{(1-s)}{s}. \] Note we
assume that $s\geq \delta>0$ by Proposition \ref{P-3-2}.
Contradiction if $a$ is sufficiently big. It follows that for some
big $a$
\[\Phi\leq\Psi_{a}.
\]

If there is some point $p=(x_0, t_0)$ such that $\Phi(x_0,
t_0)<\Psi_{-a}(x_0, t_0)$, then $\Phi-\Psi_{-a}$ obtains its
minimum interior at $q$. At the point $q$, $D^2(\Phi-\Psi_{-a})$
is nonnegative. In particular at the point $q$ $\Phi_{tt}\geq
{\Psi_{-a}}_{,tt}=2a, \triangle \Phi(q)\geq \triangle
\Psi_{-a}=t\triangle \phi_1+(1-t)\triangle \phi_0 .$ Suppose at
point $q$ \[ Q(D^2\Phi)=\Phi_{tt}(1+\triangle
\Phi)-\Phi_{tk}^2\geq Q(D^2\Psi_{-a}),\]
where\[Q(D^2\Psi_{-a})=2a(1+t\triangle \phi_1+(1-t)\triangle
\phi_0)-\left|\nabla \phi_0-\nabla \phi_1\right|^2>0
\] when $a$ is big enough.
Then \begin{eqnarray*}P(s, D^2\Phi)(q)&=&sQ(D^2\Phi)+(1-s)
(\Phi_{tt}+1+\triangle \Phi)\\
&\geq&sQ(D^2\Psi_{-a})+(1-s) (2a+t\triangle \phi_1+(1-t)\triangle
\phi_0)>1
\end{eqnarray*}
when $a$ is big enough. Contradiction. It follows that at $q$,
\begin{equation}\label{E-3-6}
Q(D^2\Phi)<Q(D^2\Psi_{-a}).
\end{equation}
Let $A$ be a $(n+2)\times(n+2)$ symmetric matrix such that the
$(n+1)\times(n+1)$ block of $A$ is $D^2\Psi_{-a}$, and
$A_{i(n+2)}=A_{(n+2)i}=0$ for $1\leq i\leq n+1$,
$A_{(n+2)(n+2)}=1$. Let $B$ be a $(n+2)\times(n+2)$ symmetric
matrix such that the $(n+1)\times(n+1)$ block of $B$ is $D^2\Phi$
and $B_{i(n+2)}=B_{(n+2)i}=0$ for $1\leq i\leq n+1$,
$B_{(n+2)(n+2)}=\lambda$. $\lambda$ is a constant satisfying
\[
Q(B)=\Phi_{tt}(\lambda+\triangle
\Phi)-\Phi_{tk}^2=Q(A)=Q(D^2\Psi_{-a}).
\]
By (\ref{E-3-1}) we know that $\lambda>1$. It follows from Lemma
\ref{L-3-2} that $Q(B-A)<0$. But $B-A$ is semi-positive definite,
$Q(B-A)\geq0.$ Contradiction.
\end{proof}
\subsection{$C^{1}$ estimates}
 At any point $p\in X\times
[0, 1]$, take local coordinates $( x_1, \cdots, x_n, t)$. We can
always choose a coordinates such that the metric tensor $g$
satisfies $g_{ij}=\delta_{ij}, \p_kg_{ij}=0$ at one point. We will
also use, for any smooth function $f$ on $X\times [0, 1]$, the
following notations
\[
\triangle f_i=\triangle (f_i), ~~\triangle f_{ij}=\triangle
(f_{ij}), ~~\triangle f_{,i}=(\triangle f)_{, i}~~\mbox{and}~~
\triangle f,_{ij}=(\triangle f)_{ij}.
\]
By Weitzenbock's formula, we have
\begin{equation}\label{W-1}
\triangle f_i=\triangle f,_{i}+R_{ij}f_j,
\end{equation}
where $R_{ij}$ is the Ricci tensor of the metric $g$. The
derivatives $f_i, f_{ij}$ etc are all covariant derivatives.
\begin{lemma}\label{L-4-3}Assume $(X, g)$ has nonnegative Ricci curvature. If
$\Phi$ is a solution of (\ref{E-3-1}), then $\Phi$ satisfies the
following {\it a priori} estimates
\[
|\nabla \Phi|\leq C,~~|\Phi_{t}|\leq C,
\]
where $C$ is a universal constant, independent of $s$ and
$\epsilon$.
\end{lemma}
\begin{proof}
Note that
\[
\Phi_{tt}+\frac{(1-s)}{s}>0.
\]
By Proposition \ref{P-3-2}, we can assume that $s\geq \delta.$ It
follows that
\[\Phi_t+\frac{(1-s) t}{s}
\] is a $t$-increasing function. It implies that
\[
\Phi_t(x, 0)-\frac{(1-s)t}{s}<\Phi_{t}<\Phi_t(x, 1)+\frac{(1-s)
}{s}-\frac{(1-s) t}{s}.
\]
Namely
\[
|\Phi_{t}|\leq \max_{\p (X\times [0, 1])}|\Phi_t|+C.
\]
On the boundary, by the $C^0$ estimates in Proposition
\ref{P-4-2}, we have
\begin{eqnarray*}
\left|\Phi_{t}(x, 0)\right|&=&\left|\lim_{t\rightarrow
0}\frac{\Phi(x, t)-\Phi(x, 0)}{t}\right|\\
&\leq&\lim_{t\rightarrow 0}\left|\frac{at(1-t)}{t}\right|+|\phi_1-\phi_0|\\
&=&a+|\phi_1-\phi_0|.
\end{eqnarray*}
Similarly, one can bound $\Phi_{t}(x, 1)$ by $a+|\phi_1-\phi_0|$,
where $a$ is the same constant as in Proposition \ref{P-4-2}.\\

To bound $|\nabla \Phi|,$ take \[h=\frac{1}{2}|\nabla \Phi|^2.
\]

Taking derivative, we get
\begin{eqnarray}\label{E-3-2}
h_t&=&\Phi_{tk}\Phi_k, ~~~
h_k=\Phi_{ik}\Phi_i\nonumber\\
h_{tt}&=&\Phi_{ttk}\Phi_k+\Phi_{tk}^2,~~~
h_{tk}=\Phi_{tik}\Phi_i+\Phi_{ti}\Phi_{ik}\nonumber\\
\triangle h&=& \Phi_{ikk}\Phi_i+\Phi^2_{ik}=\triangle \Phi_i\Phi_i+\Phi^2_{ik}\nonumber\\
&=&\triangle \Phi_{, i}\Phi_i+\Phi^2_{ik}+R_{ij}\Phi_i\Phi_j,
\end{eqnarray}
where $R_{ij}$ is the Ricci curvature of $(X, g)$. If $\Phi$
solves the equation (\ref{E-3-1}), by taking derivative, we can
get that
\begin{eqnarray}\label{e-4-4}
(s\Phi_{tt}+(1-s))\triangle \Phi_t+(s(1+\triangle
\Phi)+(1-s))\Phi_{ttt}-2s\Phi_{tk}\Phi_{ttk}=0,\\
\label{e-4-5} (s\Phi_{tt}+(1-s))\triangle \Phi_{,
k}+(s(1+\triangle \Phi)+(1-s))\Phi_{ttk}-2s\Phi_{ti}\Phi_{tik}=0.
\end{eqnarray}It follows that
\begin{eqnarray}\label{e-4-6}
dP(h)&=&(s\Phi_{tt}+(1-s))\triangle h+(s(1+\triangle
\Phi)+(1-s))h_{tt}-2s\Phi_{tk}h_{tk}\nonumber\\
&=&(s\Phi_{tt}+(1-s))\left(\triangle \Phi_{,
i}\Phi_i+\Phi^2_{ik}\right)\nonumber\\&& +(s(1+\triangle
\Phi)+(1-s))\left(\Phi_{ttk}\Phi_k+\Phi_{tk}^2\right)\nonumber\\
&&-2s\Phi_{tk}\left(\Phi_{tik}\Phi_i+\Phi_{ti}\Phi_{ik}\right)+(s\Phi_{tt}+(1-s))R_{ij}\Phi_i\Phi_j.
\end{eqnarray}
By (\ref{e-4-4}), (\ref{e-4-5}) and (\ref{e-4-6}), we have
\begin{eqnarray}\label{e-4-7}dP(h)&=&(s\Phi_{tt}+(1-s))\Phi^2_{ik}+(s(1+\triangle
\Phi)+(1-s))\Phi_{tk}^2\nonumber\\
&&-2s\Phi_{tk}\Phi_{ti}\Phi_{ik}+(s\Phi_{tt}+(1-s))R_{ij}\Phi_i\Phi_j.
\end{eqnarray}
If $|\nabla \Phi|^2$ achieves its maximum value in the interior,
we can assume that
\[\max |\nabla \Phi|^2=\max_{\p X\times [0,
1]}|\nabla \Phi|^2+\eta,\] where $\eta>0$ . Then $h+\lambda t^2$
takes its maximum in the interior for any positive constant
$\lambda\ll \eta$. Assume  the point is $p$. At point $p$,
$D^2(h+\lambda t^2)\leq 0.$ It implies that $dP(h+\lambda t^2)\leq
0$. On the other hand, by (\ref{e-4-7}), we have
\begin{eqnarray*}
&&dP(h+\lambda t^2)=(s\Phi_{tt}+(1-s))\Phi^2_{ik}+(s(1+\triangle
\Phi)+(1-s))\Phi_{tk}^2\nonumber\\&&~~~-2s\Phi_{tk}\Phi_{ti}\Phi_{ik}
+(s\Phi_{tt}+(1-s))R_{ij}\Phi_i\Phi_j+2\lambda(s(1+\triangle \Phi)+(1-s))\nonumber\\
&&~~~~~~~~~>0.
\end{eqnarray*}  Contradiction. It implies that $|\nabla \Phi|^2$ obtains its
maximum on the boundary. By the boundary condition (\ref{3-2}),
we know that $|\nabla \Phi|$ is uniformly bounded.
\end{proof}
\subsection{$C^2$ estimates}
First we have the following interior $C^2$ estimates.
\begin{lemma}\label{L-4-4}Assume $(X, g)$ has nonnegative Ricci curvature. If
$\Phi$ is a solution of (\ref{E-3-1}), then $\Phi$ satisfies the
following {\it a priori} estimate
\[
-C\leq 1+\triangle \Phi\leq C,~~-C\leq \Phi_{tt}\leq \max_{\p
(X\times [0, 1])}|\Phi_{tt}|+C,
\]
where $C$ is a universal constant, independent of $s$ and
$\epsilon$.
\end{lemma}
\begin{proof} By Proposition \ref{P-3-3}, we just need
to show the upper bound. If $1+\triangle \Phi$ achieves its
maximum in the interior, at the point $p$, we can assume that
\[
1+\triangle \Phi(p)=1+\max_{\p (X\times [0,1])}\triangle
\Phi+\eta,
\]
where $\eta>0.$ For any positive constant $\lambda\ll \eta$,
$1+\triangle \Phi+\lambda t^2$ achieves its maximum value in the
interior, at the point $\tilde p$. Then $D^2(1+\triangle
\Phi+\lambda t^2)(\tilde p)\leq 0.$ It follows that
$dP(1+\triangle \Phi+\lambda t^2)(\tilde p)\leq 0$. Let
$h=1+\triangle \Phi+\lambda t^2$. Taking derivative, we have
\begin{eqnarray}\label{e-4-8}
h_t&=&\triangle \Phi_t+2\lambda t,~~~
h_k=\triangle \Phi_{,k}\nonumber\\
h_{tt}&=&\triangle \Phi_{tt}+2\lambda,~~~
h_{tk}=\triangle\Phi_{,tk},~~~ \triangle h=\triangle ^2\Phi.
\end{eqnarray}
Note at the point $\tilde p$, $h_t=\triangle \Phi_t+2\lambda t=0,
h_k=\triangle \Phi_{, k}=0$. Taking derivative of (\ref{e-4-5}),
we have, at the point $\tilde p$,
\begin{equation*}
(s\Phi_{tt}+(1-s))\triangle^2 \Phi+(s(1+\triangle
\Phi)+(1-s))\triangle \Phi_{tt}-2s\Phi_{ti}\triangle
(\Phi_{ti})-2s\Phi_{tik}^2=0.
\end{equation*}
It follows from above and (\ref{e-4-8}) that
\begin{eqnarray*}
dP(h)(\tilde p)&=&(s\Phi_{tt}+(1-s))\triangle h+(s(1+\triangle
\Phi)+(1-s))h_{tt}-2s\Phi_{tk}h_{tk}\nonumber\\
&=&(s\Phi_{tt}+(1-s))\triangle^2\Phi\\&&~~+(s(1+\triangle
\Phi)+(1-s))(\triangle
\Phi_{tt}+2\lambda)-2s\Phi_{tk}\triangle\Phi_{,tk}\nonumber\\
&=&2\lambda(s(1+\triangle
\Phi)+(1-s))+2s\Phi_{tik}^2+2sR_{ij}\Phi_{ti}\Phi_{tj}>0.
\end{eqnarray*}
Contradiction. To bound $\Phi_{tt}$, consider
\[
h=\Phi_{tt}+1+\triangle \Phi.
\]
If $h$ obtains its maximum on the boundary, we are done. If $h$
obtains its maximum in the interior, at the point $p$, we can
assume that
\[
h(p)=\max_{\p X\times [0, 1]}h+\eta,
\]
where $\eta$ is positive. Take $\tilde h=h+\lambda t^2$ for some
positive constant $\lambda\ll \eta$, then $\tilde h$ achieves its
maximum in the interior, say, at the point $\tilde p$. It follows
that $D^2\tilde h(\tilde p)\leq 0, dP(\tilde h)(\tilde p)\leq 0$.
Taking derivative, we have
\begin{eqnarray}\label{E-3-11}
\tilde h_t&=&\Phi_{ttt}+\triangle \Phi_t+2\lambda t,~~~ \tilde h_{tt}=\Phi_{tttt}+\triangle \Phi_{tt}+2\lambda\nonumber\\
\tilde h_k&=&\Phi_{ttk}+\triangle \Phi,_{k},~~~ \triangle \tilde
h=\triangle \Phi_{tt}+\triangle ^2\Phi,~~~ \tilde
h_{tk}=\Phi_{tttk}+\triangle \Phi_{, tk}.
\end{eqnarray}
We calculate
\begin{eqnarray}\label{e-4-10}
dP(\tilde h)&=&(s\Phi_{tt}+(1-s))\triangle \tilde h+(s(1+\triangle
\Phi)+(1-s))\tilde
h_{tt}-2s\Phi_{tk}\tilde h_{tk}\nonumber\\
&=&(s(1+\triangle \Phi)+(1-s))(\Phi_{tttt}+\triangle
\Phi_{tt}+2\lambda)\nonumber\\
&&+(s\Phi_{tt}+(1-s))(\triangle \Phi_{tt}+\triangle
^2\Phi)-2\Phi_{tk}(\Phi_{tttk}+\triangle \Phi_{, tk}).
\end{eqnarray}
Taking derivative of (\ref{e-4-4}) and (\ref{e-4-5}), we have
\begin{eqnarray}\label{e-4-11}
(s\Phi_{tt}+(1-s))\triangle \Phi_{tt}+(s(1+\triangle
\Phi)+(1-s))\Phi_{tttt}-2s\Phi_{tk}\Phi_{tttk}\nonumber\\
+2s\Phi_{ttt}\triangle \Phi_t-2s\Phi_{ttk}^2=0 \\
\label{e-4-12} (s\Phi_{tt}+(1-s))\triangle^2 \Phi+(s(1+\triangle
\Phi)+(1-s))\triangle
\Phi_{tt}-2s\Phi_{ti}\Phi_{tikk}\nonumber\\
+2s\Phi_{ttk}\triangle\Phi_{,k}-2s\Phi_{tik}^2 =0.
\end{eqnarray}
It follows from (\ref{e-4-10}), (\ref{e-4-11}) and (\ref{e-4-12})
that
\begin{eqnarray*}
dP(\tilde h)&=&2\lambda(s(1+\triangle
\Phi)+(1-s))+2s\Phi_{ttk}^2+2s\Phi_{tik}^2\\
&&-2s\Phi_{ttt}\triangle \Phi_t-2s\Phi_{ttk}\triangle
\Phi_{,k}+2sR_{ij}\Phi_{ti}\Phi_{tj}.
\end{eqnarray*}
Denote \begin{equation}\label{e-3-19}
L=\Phi_{ttk}^2-\Phi_{ttt}\triangle \Phi_t,
M=\Phi_{tij}^2-\Phi_{ttk}\triangle \Phi_{, k}.
\end{equation}
Using \eqref{e-4-4} and (\ref{e-4-5}), we can get that
\begin{eqnarray*}
\left(s\Phi_{tt}+(1-s)\right)L&=&
(s\Phi_{tt}+(1-s))\Phi_{ttk}^2+(s(1+\triangle\Phi)+(1-s))\Phi_{ttt}^2\\&&-2s\Phi_{tk}\Phi_{ttk}\Phi_{ttt},\\
\left(s\Phi_{tt}+(1-s)\right)M
&=&(s\Phi_{tt}+(1-s))\Phi_{tij}^2+(s(1+\triangle\Phi)+(1-s))\Phi_{ttk}^2\\&&-2s\Phi_{tk}\Phi_{tik}\Phi_{tti}.
\end{eqnarray*}
It follows that $L, M\geq 0$. Hence,
\begin{eqnarray*}
dP(\tilde h)&\geq&2\lambda(s(1+\triangle \Phi)+(1-s))>0.
\end{eqnarray*}
Contradiction.
\end{proof}

Then we derive the boundary estimates of $\triangle \Phi, \Phi_{tt}$ and $\nabla \Phi_t$ by careful construction of some barrier functions. This type of construction of barrier functions follows from \cite{Guan-Spruck} and \cite{Guan}. 
\begin{lemma}\label{L-4-5}Assume $(X, g)$ has nonnegative Ricci curvature. If
$\Phi$ is a solution of (\ref{E-3-1}), then $\Phi$ satisfies the
following {\it a priori} estimate
\[
|\triangle \Phi|\leq C~~, |\Phi_{tk}|\leq C,~~ |\Phi_{tt}|\leq C,
\]
where $C$ is a universal constant, independent of $s$ and
$\epsilon$. Note that the proof of boundary estimates do not depend on
the Ricci curvature assumption if we get the interior $C^2$
estimates.
\end{lemma}
\begin{proof}In light of Proposition \ref{P-3-3}, we assume $s\geq\delta .$ By Lemma \ref{L-4-4}, $\triangle \Phi$ is uniformly
bounded. If $\Phi$ solves equation (\ref{E-3-1}), we have
\begin{eqnarray*}
|\Phi_{tk}|^2&=&\Phi_{tt}(1+\triangle
\Phi)+\frac{(1-s)}{s}(\Phi_{tt}+\triangle
\Phi)-\frac{\epsilon }{s}\\
&\leq&C(|\Phi_{tt}|+1)\\
&\leq&C\left(\max_{\p (X\times [0, 1])}|\Phi_{tt}|+1\right).
\end{eqnarray*}
To bound $\max_{\p (X\times [0, 1])}|\Phi_{tt}|$, observe that
\[
(s(1+\triangle \Phi)+(1-s))\Phi_{tt}=s|\Phi_{tk}|^2-(1-s)
\triangle \Phi+\epsilon.
\]
Since $\Phi(x, 0)=\phi_0, ~\Phi(x, 1)=\phi_1$, $1+\triangle \Phi$
is positive and uniformly bounded away from zero on the boundary.
Thus $s(1+\triangle \Phi)+(1-s)$ is uniformly bounded away from
zero on the boundary. It follows that
\[
\max_{\p (X\times [0, 1])}|\Phi_{tt}|\leq C\left(\max_{\p (X\times
[0, 1])}|\Phi_{tk}|^2+1\right).
\]
We finish the proof by showing \[\max_{\p (X\times [0,
1])}|\Phi_{tk}|\] is bounded. We just consider $|\Phi_{tk}|$ on
$X\times \{0\}$. The case on $X\times \{1\}$ follows similarly.
For any point $p$ on $X\times \{0\}$, we can choose a local
coordinates around $p$ such that $p=(0, \cdots, 0, 0)$, and
\[
g_{ij}(0)=\delta_{ij}, ~~~\p g_{ij}(0)=0.
\]
Let $B_0(\rho)\subset X$ be a small ball with radius $\rho.$ For
any $x\in B_{0}(\rho)$, we have
\[
g_{ij}(x)=(1+o(|x|))\delta_{ij}, ~~~\p g_{ij}(x)=o(|x|).
\]
Take $\Omega=B_{0}(\rho)\times [0, \kappa],$ where $\kappa$ is a
small positive number. Let $h$ be a function on $\Omega$ defined
by
\[ h=(\Phi-\phi_0)_{,k}
\]
for any $k=1, 2, \cdots, n$. Note $h$ is only a local defined
function. Take
\[
\tilde h=A(\Phi-\phi_0+At-At^2)+B(\sum x_i^2)+h,
\]
where $A\gg B$ are two big fixed positive constants. Since $h$ is
bounded, it is easy to see that $\tilde h\geq 0$ on $\p \Omega.$
Taking derivative, we have
\begin{eqnarray*}
h_{t}&=&\Phi_{tk},~~ h_{i}=\Phi_{ki}-(\phi_0)_{ki},~~
h_{tt}=\Phi_{ttk}\\
h_{ti}&=&\Phi_{tki},~~~ \triangle h=\triangle \Phi_k-{\triangle
\phi_0}_k.
\end{eqnarray*}
It follows from above and (\ref{e-4-5}) that
\begin{eqnarray*}
dP(h)&=&(s\Phi_{tt}+(1-s))\triangle h+(s(1+\triangle
\Phi)+(1-s))h_{tt}-2s\Phi_{ti}h_{ti}\nonumber\\
&=&(s\Phi_{tt}+(1-s))\triangle \Phi_k+(s(1+\triangle
\Phi)+(1-s))\Phi_{ttk}\nonumber\\
&&-2s\Phi_{ti}\Phi_{tik}-(s\Phi_{tt}+(1-s)){\triangle\phi_0}_k\nonumber\\
&=&(R_{kl}\Phi_l-{\triangle \phi_0}_k)(s\Phi_{tt}+(1-s))\nonumber\\
&\leq&C(s\Phi_{tt}+(1-s)).
\end{eqnarray*}
We can calculate
\begin{eqnarray*}
dP(\Phi)&=&(s\Phi_{tt}+(1-s))\triangle \Phi+(s(1+\triangle
\Phi)+(1-s))\Phi_{tt}-2s\Phi_{tk}^2\\
&=&2\epsilon-(1-s)(\Phi_{tt}+\triangle \Phi)-s\Phi_{tt}.
\end{eqnarray*}
It follows that
\[
dP(\Phi-\phi_0)=-(1+\triangle
\phi_0)(s\Phi_{tt}+(1-s))-(1-s)(\Phi_{tt}+1+\triangle
\Phi)+2\epsilon,
\]
and
\[dP(-At^2)=-2A(s(1+\triangle \Phi)+(1-s)).
\]
Since on $X\times \{0\}$, $s(1+\triangle \Phi)+(1-s)$ is uniformly
bounded away from zero, we can pick up $\kappa$ small enough such
that on $X\times [0, \kappa]$,
\[
A(s(1+\triangle \Phi)+(1-s))>2\epsilon.
\]
We should emphasize that $A, B$ are uniformly bounded even though
$\kappa$ might depend on $s$.  It follows that
\[dP(\Phi-\phi_0+At-At^2)<-(1+\triangle
\phi_0)(s\Phi_{tt}+(1-s)).
\]
And
\[\triangle (\sum x_i^2)=2n+o(\rho)< 2n+1,
\]
we can calculate
\[dP(\sum x_i^2)<(2n+1)(s\Phi_{tt}+(1-s)).
\]
It follows from above that in $\Omega$,
\[
dP(\tilde h)<\left(-A(1+\triangle
\phi_0)+(2n+1)B+C\right)(s\Phi_{tt}+(1-s))<0.
\]
By the maximum principle, we know that $\tilde h\geq 0$ in
$\Omega$. Since $\tilde h(0)=0$,
\[\frac{\p \tilde h}{\p t}|_{t=0}\geq 0.
\]
It follows that on $X\times \{0\}$,
\[\Phi_{tk}\geq -C.
\]
It follows similarly that
\[
\Phi_{tk}\leq C,
\] by taking $h=(\phi_0-\Phi)_k$, and
\[
\tilde h=A(\Phi-\phi_0+At-At^2)+B(\sum x_i^2)+h.
\]
\end{proof}

Hence we have proved the weakly $C^2$ bound of $\Phi$ for any smooth solution of \eqref{E-3-1}. When $\epsilon>0$, for any $s\in [0, 1]$, the equation \eqref{E-3-1} is uniformly elliptic.   We can rewrite the equation  for $s>0$
\begin{equation}\label{concave}
\log \left(\left(\Phi_{tt}+\frac{(1-s)}{s}\right)\left(\triangle\Phi+1+\frac{(1-s)}{s}\right)-|\nabla
\Phi_t|^2\right)=\log\frac{s\epsilon+(1-s)^2}{s^2}.
\end{equation}
It is clear that the equation is concave for any $s\in [0, 1]$ and $\epsilon>0$ by Lemma \ref{L-3-3}. It then  follows from   the Evans-Krylov theory to get that the H\"older estimates of $D^2\Phi$. In fact, the assumption of $\Phi\in C^{1, \beta}$ for some $\beta \in (0, 1)$ is sufficient to get global $C^{2, \alpha}$ regularity for a uniformly elliptic and concave fully nonlinear equation provided sufficient smooth boundary data and the right hand side (see Theorem 7.3 in \cite{ChenWu} for example). 

\begin{rem}In a previous version, first we prove that $|D^2\Phi|$ is bounded by  using the weak maximum principle (c.f. Gilbarg-Trudinger \cite{Giltru1998}, Section 9.7, Theorem 9.20 and Section 9.9 Theorem 9.26) and then use Evans-Krylov theory to obtain H\"older estimate of $|D^2\Phi|$. We would like to thank the referee for pointing  out that this not necessary for using Evans-Krylov theory. \end{rem}

Once we have the H\"older estimates of $D^2\Phi$, the  Schauder theory
gives all the higher derivatives bound. To see this, if $\Phi$
solves (\ref{E-3-1}), then we have \[dP(\Phi_t)=0.\] Since the
coefficients are all H\"older continuous, the Schauder
theory applies and we get that $\Phi_t$ is $C^{2, \alpha}$. The
similar discussion holds for $\Phi_{k}$, namely,
\[
dP(\Phi_k)=R_{ki}\Phi_i.
\]
Then the boot-strapping argument allows us to conclude that all
higher derivatives are bounded and $\Phi\in C^{\infty}$ follows.
We should emphasize that the above discussion holds only for
$\epsilon>0$. Summarize above, we have
\begin{lemma}\label{L-13}
If $\Phi$ solves the equation (\ref{E-3-1}), then the following
estimates hold independent of $s$,
\[
|D^l\Phi|\leq C=C(\epsilon, l)
\]
for any $l\in \mathbb{N}$.
\end{lemma}
We are in the position to prove Theorem \ref{T-3-1}.
\begin{proof}First we show that $s_0=1$. Otherwise, assume
$s_0<1$. Then the continuous family (\ref{E-3-1}) has a unique
solution for $0\leq s<s_0$. Consider a sequence of $s_i\rightarrow
s_0$ and the solutions $\Phi^{i}$. In light of Proposition
\ref{P-3-3} we can assume $s_i\geq \delta$. By the {\it a priori }
estimates derived above, for any $s_i$, the solutions $\Phi^{i}$
of (\ref{E-3-1}) satisfy
\[
|\Phi^{i}|, |\nabla \Phi^{i}|, |\Phi^{i}_t|, |\Phi^{i}_{tk}|,
|\Phi^{i}_{tt}|, |\triangle \Phi^{i}|\leq C.
\]
In particular, $dP$ is uniformly elliptic for any $s_i$. When
$\epsilon>0$, we still have the global H\"older estimates for
second derivatives. Then by the standard
elliptic argument, we obtain the higher derivatives estimates for
$\Phi^{i}.$ Then the solutions $\Phi^{i}$ converge to a smooth
solution for $s=s_0.$ Again $dP$ is a uniform elliptic operator at
$s=s_0$ and $dP$ has zero kernel with zero boundary data, one can
solve the equation (\ref{E-3-1}) for $s$ is sufficiently close to
$s_0$, but this contradicts with the definition of $s_0$. So
$s_0=1$. Moreover, the solutions $\Phi^{i}$ satisfy the estimates
\[
|\Phi^{i}|, |\nabla \Phi^{i}|, |\Phi^{i}_t|, |\Phi^{i}_{tk}|,
|\Phi^{i}_{tt}|, |\triangle \Phi^{i}|\leq C.
\]
The standard elliptic theory gives the higher derivatives
estimates.  So $\Phi^{i}$ sub-converge to a smooth function
$\Phi(x, t)$ solving the equation
\[
Q(D^2\Phi)=\Phi_{tt}(1+\triangle \Phi)-|\nabla
\dot{\Phi}|^2=\epsilon.
\]
In particular, we know that $\Phi_{tt}+1+\triangle \Phi$ is
bounded away from zero. When $s=1$, it follows that $\Phi_{tt},
1+\triangle \Phi>0$ and $\Phi\in \cH$.
\end{proof}

\section{Without assumption on Ricci}
In this section we drop the non-negative assumption of Ricci to
prove Theorem \ref{T-3-1}. Note that $C^{0}$ estimates, the boundary
$C^{1}, C^{2}$ estimates, and higher derivative estimates do not
depend on the Ricci curvature assumption.  We just need to
establish interior $C^{1}$ and  $C^{2}$ estimates.
\subsection{$C^1$ estimate}
\begin{lemma}If
$\Phi$ is a solution of (\ref{E-3-1}), then $\Phi$ satisfies the
following {\it a priori} estimates
\[
|\nabla \Phi|\leq C,~~|\Phi_{t}|\leq C,
\]
where $C$ is a universal constant, independent of $s$ and
$\epsilon$.
\end{lemma}
\begin{proof}
Note that $|\Phi_t|$ is uniformly bounded without Ricci assumption. If
$\Phi$ solves (\ref{E-3-1}) with boundary condition (\ref{E-3-2}),
then $\tilde \Phi=\Phi+at$ solves (\ref{E-3-1}) with  boundary
condition
\[
\tilde \Phi(x, 0)=\phi_0,~~\tilde\Phi(x, 1)=\phi_1+A,
\]
where $A$ is any constant. Thus we can assume that
\[\Phi_{t}^2\gg\Phi>0
\]
by choosing the normalizing condition on $\phi_0, \phi_1$.  Take
\[
h=\frac{1}{2}\left(|\nabla \Phi|^2+b\Phi^2\right).
\]
One can calculate that
\begin{eqnarray*}
\frac{1}{2}dP(|\nabla
\Phi|^2)&=&(s\Phi_{tt}+(1-s))\Phi^2_{ik}+(s(1+\triangle
\Phi)+(1-s))\Phi_{tk}^2-2s\Phi_{tk}\Phi_{ti}\Phi_{ik}\\
&&+(s\Phi_{tt}+(1-s))R_{ij}\Phi_i\Phi_j.\\
\end{eqnarray*}
And
\begin{eqnarray*}
dP(\frac{\Phi^2}{2})&=&(s\Phi_{tt}+(1-s))\triangle
\left(\frac{\Phi^2}{2}\right)+(s(1+\triangle
\Phi)+(1-s))\left(\frac{\Phi^2}{2}\right)_{tt}-2s\Phi_{tk}\left(\frac{\Phi^2}{2}\right)_{tk}\\&=&
(s\Phi_{tt}+(1-s))(\triangle \Phi \Phi+|\nabla
\Phi|^2)+(s(1+\triangle
\Phi)+(1-s))(\Phi_{tt}\Phi+\Phi_t^2)\\&&-2s\Phi_{tk}(\Phi\Phi_{tk}+\Phi_t\Phi_k)\\
&=&(s\Phi_{tt}+(1-s))\Phi\triangle \Phi+(s(1+\triangle
\Phi)+(1-s))\Phi\Phi_{tt}-2s\Phi\Phi_{tk}^2\\
&&+(s\Phi_{tt}+(1-s))|\nabla \Phi|^2+(s(1+\triangle
\Phi)+(1-s))\Phi_t^2-2s\Phi_{tk}\Phi_t\Phi_k\\
&=&2\epsilon\Phi-\Phi(s\Phi_{tt}+(1-s))-(1-s)\Phi(\Phi_{tt}+\triangle
\Phi)\\
&&+(s\Phi_{tt}+(1-s))|\nabla \Phi|^2+(s(1+\triangle
\Phi)+(1-s))\Phi_t^2-2s\Phi_{tk}\Phi_t\Phi_k.
\end{eqnarray*}

It follows that
\begin{eqnarray}\label{e-6-1}
dP(h)&=&(s\Phi_{tt}+(1-s))\Phi^2_{ik}+(s(1+\triangle
\Phi)+(1-s))\Phi_{tk}^2-2s\Phi_{tk}\Phi_{ti}\Phi_{ik}\nonumber\\
&&+(s\Phi_{tt}+(1-s))R_{ij}\Phi_i\Phi_j+b(s\Phi_{tt}+(1-s))|\nabla
\Phi|^2\nonumber\\&&-b\Phi((s\Phi_{tt}+(1-s))+(1-s)(\Phi_{tt}+\triangle
\Phi))\nonumber\\
&&+b(s(1+\triangle \Phi)+(1-s))\Phi_t^2-2bs\Phi_{tk}\Phi_t\Phi_k.
\end{eqnarray}
We want to show that
\[
h\leq \max_{\p (X\times [0, 1])}h+C,
\]
where $C$ is a universal constant. If not,  $h$ obtains its
maximum in the interior, at the point $p$. Note at the point $p$,
\[
h_t=\Phi_{tk}\Phi_k+b\Phi_t\Phi=0.
\]
By (\ref{e-6-1}), we get at the point $p$,
\begin{eqnarray}\label{e-6-2}
dP(h)&=&(s\Phi_{tt}+(1-s))\Phi^2_{ik}+(s(1+\triangle
\Phi)+(1-s))\Phi_{tk}^2-2s\Phi_{tk}\Phi_{ti}\Phi_{ik}\nonumber\\
&&+(s\Phi_{tt}+(1-s))R_{ij}\Phi_i\Phi_j+b(s\Phi_{tt}+(1-s))|\nabla
\Phi|^2\nonumber\\&&-b\Phi((s\Phi_{tt}+(1-s))+(1-s)(\Phi_{tt}+\triangle
\Phi))\nonumber\\
&&+b(s(1+\triangle \Phi)+(1-s))\Phi_t^2+2b^2s\Phi\Phi_t^2.
\end{eqnarray}
Since at the point $p$, $D^2h\leq 0$, then $dP(h)(p)\leq 0.$ By
(\ref{e-6-2}), we have  \begin{eqnarray*}
dP(h)&>&\left((b-C)|\nabla
\Phi|^2-b\Phi\right)(s\Phi_{tt}+(1-s))\\&&~~~+b\Phi_t^2(s(1+\triangle
\Phi)+(1-s))-b\Phi(1-s)(\Phi_{tt}+\triangle \Phi).
\end{eqnarray*}
We assume $s\geq \delta$ and $\Phi^2_t\gg \Phi>0.$ Pick constant
$b$ big enough, and if $|\nabla \Phi|^2$ is too big, we have
\[dP(h)>0.
\]
Contradiction.
\end{proof}

\subsection{$C^2$ estimate}
We have the following interior estimates
\begin{lemma}If
$\Phi$ is a solution of (\ref{E-3-1}), then $\Phi$ satisfies the
following {\it a priori} estimate
\[
0<\triangle \Phi+1+ \Phi_{tt}\leq C\left(\max_{\p (X\times [0,
1])}|\Phi_{tt}|+1\right),
\]
where $C$ is a universal constant, independent of $s$ and $\epsilon$.
\end{lemma}
\begin{proof}
Denote
\[
f=\Phi_{tt}+1+\triangle \Phi, ~~~h=\frac{1}{2}b t^2-b\Phi+A,
\]
where $b, A$ are positive constant such that $h>0.$ Take \[ \tilde
h=f\exp (h).
\]
To get the interior estimates we want to show that
\begin{equation}\label{e-6-3}
\tilde h\leq \max_{\p X\times[0, 1]}\tilde h+C \end{equation} for
some uniformly bounded constant $C$. Note that if (\ref{e-6-3})
holds then one can proceed to prove the boundary estimate saying
that $f$ is actually uniformly bounded. If (\ref{e-6-3}) does not
hold,  $\tilde h$  obtains its maximum in the interior, at the
point $p$. It follows that $D^2\tilde h(p)\leq 0$ and $dP(\tilde
h)(p)\leq 0.$ Taking derivative,
\[
\tilde h_t=\exp(h)(fh_t+f_t), ~\tilde h_{k}=\exp(h)(f_{k}+fh_k),
\]
and
\[
\tilde h_{tt}=\exp(h)(fh_t^2+2f_th_t+fh_{tt}+f_{tt}),~ \tilde
h_{kk}=\exp(h)(fh_k^2+2f_kh_k+fh_{kk}+f_{kk}).
\]
Also we have
\[
\tilde h_{tk}=\exp(h)(f_kh_t+f_th_k+fh_th_k+fh_{tk}).
\]
Note also at the point $p$, we have
\[
\tilde h_t=0, ~~\tilde h_k=0.
\]
It follows that
\[
fh_t+f_t=0, ~~fh_k+f_k=0.
\]
One can calculate that at the point,
\begin{eqnarray}\label{e-6-4}
dP(\tilde h)&=&(s\Phi_{tt}+(1-s))\triangle \tilde h+(s(1+\triangle
\Phi)+(1-s))\tilde
h_{tt}-2s\Phi_{tk}\tilde h_{tk}\nonumber\\
&=&(s\Phi_{tt}+(1-s))\exp(h)(fh_k^2+2f_kh_k+fh_{kk}+f_{kk})\nonumber\\
&&+(s(1+\triangle
\Phi)+(1-s))\exp(h)(fh_t^2+2f_th_t+fh_{tt}+f_{tt})\nonumber\\
&&-2s\Phi_{tk}\exp(h)(f_kh_t+f_th_k+fh_th_k+fh_{tk})\nonumber\\
&=&\exp(h)f((s\Phi_{tt}+(1-s))h_{kk}+(s(1+\triangle
\Phi)+(1-s))h_{tt}-2s\Phi_{tk}h_{tk})\nonumber\\
&&+\exp(h)((s\Phi_{tt}+(1-s))f_{kk}+(s(1+\triangle
\Phi)+(1-s))f_{tt}-2s\Phi_{tk}f_{tk})\nonumber\\
&&+(s\Phi_{tt}+(1-s))\exp(h)(fh_k^2+2f_kh_k)+(s(1+\triangle
\Phi)\nonumber\\
&&+(1-s))\exp(h)(fh_t^2+2f_th_t)-2s\Phi_{tk}\exp(h)(f_kh_t+f_th_k+fh_th_k)\nonumber\\
&=&\exp(h)\left(fdP(h)+dP(f)-Q(h, f)\right),
\end{eqnarray}
where
\begin{eqnarray}\label{e-6-5}
Q(h, f)&=&f((s\Phi_{tt}+(1-s))h_k^2+(s(1+\triangle
\Phi)+(1-s))h_t^2-2s\Phi_{tk}h_th_k).\nonumber\\
\end{eqnarray}
Now we carry out $dP(f), dP(h)$. We know that (c. f.
(\ref{e-3-19}))
\[dP(f)=s(L+M)+ 2sR_{ij}\Phi_{ti}\Phi_{tj}\geq 2sR_{ij}\Phi_{ti}\Phi_{tj}.
\]
It is easy to get that at the point $p$,
\begin{equation}\label{e-6-13}
dP(h)=b(\Phi_{tt}+1+\triangle \Phi)-2bs\epsilon\geq
b(\Phi_{tt}+1+\triangle \Phi)-2b.
\end{equation}
 By (\ref{e-6-4}) and (\ref{e-6-13}) we have
\begin{eqnarray}\label{e-6-14}dP(\tilde h)&\geq &b\exp(h)(\Phi_{tt}+1+\triangle
\Phi)^2-2b(\Phi_{tt}+1+\triangle\Phi)\nonumber\\&&~+2s\exp(h)R_{ij}\Phi_{ti}\Phi_{tj}-\exp(h)Q(h, f).
\end{eqnarray}
It is easy to see that
\[
R_{ij}\Phi_{ti}\Phi_{tj}>-C(\Phi_{tt}+1+\triangle\Phi )^2.
\]
By (\ref{e-6-5}) it is clear that
\[
Q(h, f)\leq C(\Phi_{tt}+1+\triangle \Phi)^2.
\]
It follows that
\[
dP(\tilde h)>\exp(h)\left\{(b-C)(\Phi_{tt}+1+\triangle
\Phi)^2-2b(\Phi_{tt}+1+\triangle\Phi)\right\}.
\]
We can assume that
\[
\Phi_{tt}+1+\triangle \Phi> \frac{2b}{b-C}
\]
at the point $p$. Contradiction with $dP(\tilde h)\leq 0$ at the
point $p$.
\end{proof}

\section{The space $\cH$}
In this section, we discuss the property of the space  $\cH$ by
using the weak solution of the geodesic equation. The discussion
below follows closely Chen \cite{Chen2000}, Calabi-Chen
\cite{CC2002} in the case of the space of K\"ahler metrics.

\subsection{Uniqueness of the weak solution}
For the weak solution obtained for the geodesic equation, it fits
the standard notion of  {\it viscosity solution} developed in fully
nonlinear equations. But since we
obtain a global approximation of the weak solution, we will use
the following approximation instead of  viscosity solution.
\begin{defi} A  continuous function $\Phi$ in $X\times [0, 1]$ is a
weak $C^0$ solution to the geodesic equation (\ref{E-1-1}) with
prescribed boundary data if for any $\epsilon>0$, there exists a
smooth function $\tilde \Phi\in \cH$ such that
\[
|\Phi-\tilde\Phi|<\delta=\delta(\epsilon)
\]
and $\tilde \Phi$ solves
\[
Q(D^2\tilde \Phi)=\epsilon
\] with the same boundary data, where $\delta=\delta(\epsilon)$ is
a constant depending on $\epsilon$ and $\delta\rightarrow 0$ when
$\epsilon\rightarrow 0.$
\end{defi}
Obviously, the weak solution we obtain is a weak $C^0$ solution of
the geodesic equation (\ref{E-1-1}).

\begin{theo}\label{T-5-2}Suppose $\Phi, \Psi$ are two $C^0$ weak solutions
to the geodesic equation (\ref{E-1-1}) with prescribed boundary
data $(\phi_0, \phi_1)$ and $(\psi_0, \psi_1)$. Then
\[
\max_{X\times[0, 1]}|\Phi-\Psi|\leq \max_{\p (X\times[0, 1])}
(|\phi_0-\psi_0|, |\phi_1-\psi_1|).
\]
\end{theo}

\begin{proof} For any $\epsilon>0$, we have $\tilde \Phi, \tilde
\Psi$ solving
\[
Q(D^2\tilde\Phi)=Q(D^2\tilde\Psi)=\epsilon
\]
with respect to the boundary condition. Then we know that
\[
|\Phi-\tilde\Phi|<\delta, ~~~|\Psi-\tilde\Psi|<\delta.
\]
We want to show that
\[
\max_{X\times[0, 1]}|\tilde\Phi-\tilde\Psi|\leq \max_{\p
(X\times[0, 1])} (|\phi_0-\psi_0|, |\phi_1-\psi_1|).
\]
If the boundary conditions are the same, namely, $\phi_0=\psi_0,
\phi_1=\psi_1$, we have $\tilde\Phi=\tilde\Psi$. In general, we
can use the same trick as in Proposition \ref{P-4-2} to get the
inequality above. If $\max(\tilde\Phi-\tilde\Psi)>\max_{\p
(X\times[0, 1])} (|\phi_0-\psi_0|, |\phi_1-\psi_1|)$, then
$\tilde\Phi-\tilde\Psi-\lambda t(1-t)$ obtain its maximum in
interior, where $\lambda>0$ is small enough. It follows that
$D^2\Psi>D^2\Phi$. But $Q(D^2\Phi)=Q(D^2\Psi)=\epsilon$ implies
that $Q(D^2\Psi-D^2\Phi)<0$, contradiction. Switch $\tilde\Phi,
\tilde\Psi $, we can get the inequality as promised.
\\

It follows that
\[
|\Phi-\Psi|<\max_{\p (X\times[0, 1])}
(|\tilde\phi_0-\tilde\psi_0|,
|\tilde\phi_1-\tilde\psi_1|)+2\delta.
\]
Let $\epsilon\rightarrow 0$, we get that
\[
|\Phi-\Psi|<\max_{\p (X\times[0, 1])} (|\phi_0-\psi_0|,
|\phi_1-\psi_1|).
\]
\end{proof}
As a direct consequence we have
\begin{cor}The weak solution of the geodesic equation (\ref{E-1-1})
is unique with the fixed boundary data.
\end{cor}

\subsection{$\cH$ is a metric space}
In this section we want to prove that $\cH$ is a metric space and
the weak $C^2$ geodesic realizes the global minimum of the length
over all paths. For simplicity, for any $\Phi\in \cH$, we fix the
normalization
\[
\int_M\Phi dg=0.
\]
Since the solution we obtain does not have enough derivatives, we
use the solution
\[
Q(D^2\tilde\Phi)=\epsilon
\] to approximate the weak solution.
\begin{defi}
Let $\phi_1, \phi_2$ be two points in $\cH$. Then we know there
exists a unique weak geodesic connecting these two points. Define
the geodesic distance \[d(\phi_1,
\phi_2)=\int_0^1dt\sqrt{\int_M\dot\Phi^2\left(1+\triangle
\Phi\right)dg}
\] as the length of this
geodesic, where $\Phi$ solves the geodesic equation weakly with
boundary condition $\phi_1, \phi_2$.
\end{defi}

Then we have the following
\begin{lemma}\label{L-7-6}
Suppose $\Phi(t)$ is a weak $C^2$ geodesic in $\cH$ from $0$ to
$\phi$ and we normalize $\Phi$ such that
\[
\int_M\Phi dg=0.
\]
Also we can assume that $Vol (M, g)=1$. Then we have
\[d(0, \phi)\geq\max\left(\sqrt{\int_{\phi>0}\phi^2(1+\triangle\phi)dg}, \sqrt{\int_{\phi<0}\phi^2 dg}\right).\]
\end{lemma} In other words, the length of any weak $C^2$ geodesic is strictly
positive.
\begin{proof}
Suppose $\tilde \Phi$ is the solution of
\[
Q(D^2\tilde \Phi)=\epsilon
\]
with the same boundary condition. It is easy to see that
\[
d(0, \phi)=\lim_{\epsilon\rightarrow 0}d_{\epsilon}(0, \phi),
\]
where
\[d_{\epsilon}=\int_0^1dt\sqrt{\int_M\dot{\tilde\Phi}^2\left(1+\triangle
\tilde\Phi\right)dg}.
\]
Denote the energy element as
\[
E_{\epsilon}(t)=\int_M\dot{\tilde\Phi}^2\left(1+\triangle
\tilde\Phi\right)dg.
\]
It is easy to see that
\[
\left|\frac{d}{dt}E_\epsilon(t)\right|\leq C\epsilon,
\]
where $C$ is a universal constant. It follows that
\[
\left|E_{\epsilon}(t_1)-E_{\epsilon}(t_2)\right|\leq C\epsilon
|t_1-t_2|\leq C\epsilon.
\]
In particular, we have $E(t)$ is a constant for any $t$ even
$\Phi$ has no enough derivative. Note that if $\Phi$ is smooth this is
a direct consequence of the geodesic equation. Since $\ddot{\tilde
\Phi}>0$, it follows that
\[
\dot{\tilde\Phi}(0)<\phi<\dot{\tilde \Phi}(1).
\]
It implies
\[
E_{\epsilon}(0)=\int_M\dot{\tilde\Phi}^2(0)dg>\int_{\phi<0}\phi^2dg
\] and
\[
E_{\epsilon}(1)=\int_M\dot{\tilde\Phi}^2(1)(1+\triangle\phi)dg>\int_{\phi>0}\phi^2(1+\triangle\phi)dg.
\]
Thus,
\[
d_{\epsilon}(0, \phi)>\max
\left(\sqrt{\int_{\phi<0}\phi^2dg-C\epsilon},
\sqrt{\int_{\phi>0}\phi^2(1+\triangle\phi)dg-C\epsilon}\right).
\]
So we can get
\[
d(0,
\phi)\geq\max\left(\sqrt{\int_{\phi>0}\phi^2(1+\triangle\phi)dg},
\sqrt{\int_{\phi<0}\phi^2 dg}\right).
\]

\end{proof}
We need the next geodesic approximation lemma.
\begin{lemma}\label{L-7-7}Suppose $C_i: \phi_{i}(s): [0, 1]\rightarrow \cH (i=0,
1)$ are tow smooth curves in $\cH$. For any $\epsilon>0$, there
exist two parameter families of smooth curves $\C(t, s, \epsilon):
\Phi(t, s, \epsilon)$ solving
\[
Q(D^2\Phi)=\epsilon
\]with boundary condition
\[
\Phi(0, s, \epsilon)=\phi_0(s), ~~\Phi(1, s, \epsilon)=\phi_1(s)
\]
satisfying the following:\\

1. There exists a uniformly bounded constant $C=C(M, \phi_0,
\phi_1)$ such that
\[
\left|\Phi\right|+\left|\frac{\p \Phi}{\p t}\right|+\left|\frac{\p
\Phi}{\p s}\right|\leq C; ~0<\frac{\p^2\Phi}{\p t^2}\leq C;~
\frac{\p^2\Phi}{\p s^2}\leq C.\]

2. The convex curve $C(s, \epsilon)$ converges to the unique
geodesic between $\phi_0(s)$ and $\phi_1(s)$ in the weak $C^2$
topology.\\

3. Define the energy element along $C(s, \epsilon)$ as
\[
E(t, s, \epsilon)=\int_M\left|\frac{\p \Phi}{\p
t}\right|^2(1+\triangle \Phi)dg.
\]
There exists a uniform constant $C$ \[\left|\frac{\p E}{\p
t}\right|\leq C\epsilon.\]
\end{lemma}
\begin{proof}The lemma is clear except
\[
\left|\frac{\p \Phi}{\p s}\right|\leq C,~~\frac{\p^2\Phi}{\p
s^2}\leq C.
\]
The inequalities above follow from the maximum principle directly
since
\[
dQ\left(\frac{\p \Phi}{\p s}\right)=0,~~dQ\left(\frac{\p^2\Phi}{\p
s^2}\right)\geq 0,
\]
where the last inequality is a consequence of concavity of the
operator $Q$.
\end{proof}

Next we show that $d$ is a continuous function in $\cH$. First we have

\begin{lemma}\label{L-5-7}Suppose $\phi_0\in \cH$, then for any $\phi\in \cH$, $d(\phi_0, \phi)\rightarrow 0$ when $\phi\rightarrow \phi_0$ in $C^k$ topology for $k\geq 4$. \end{lemma}
\begin{proof}This is really a consequence of Theorem \ref{T-5-2}.  $\Phi_0(t)\equiv\phi_0$ is a geodesic with length zero. Let $\Phi(t)$ be the weak $C^2$ geodesic which connects $\phi_0$ and $\phi$, then by Theorem \ref{T-5-2}, $|\Phi(t)-\Phi_0(t)|\leq |\phi_0-\phi|.$ For any $x$ fixed, apply the interpolation inequality for $\Phi(t, x)-\phi_0(x)$ in $[0, 1]$ such that, for any $\epsilon_1>0$, 
\[
\left|\frac{\p}{\p t}\left(\Phi(t, x)-\phi_0(x)\right)\right|\leq C(\epsilon_1)\max_t|\Phi(t, x)-\phi_0(x)|+\epsilon_1\max_t\left|\frac{\p^2}{\p t^2}\left(\Phi(t)-\phi_0\right)\right|.
\]
In the case of $t\in [0, 1]$ the proof of the above interpolation inequality is straightforward.
We can assume that $|\phi-\phi_0|_{C^4}\leq 1$, hence  $|\Phi_{tt}|\leq C_1$ for $C_1$ depending only on $M, g, \phi_0$.
It then follows that
\[
|\Phi_t|\leq C(\epsilon_1)|\phi(x)-\phi_0(x)|+C_1\epsilon_1,
\]
It then implies that when $|\phi-\phi_0|\rightarrow 0$, $|\Phi_t|\rightarrow 0$. Hence $d(\phi_0, \phi)\rightarrow 0$. 

\end{proof}
 
 We shall then prove the triangle inequality.
\begin{theo}\label{T-7-8}Suppose $C$ is a smooth simple curve defined by $ \phi(s): s\in[0, 1]\rightarrow \cH $.  Let $\psi\in \cH$ be a point which is not on $C$. For any $s$,
\[
d(\psi, \phi(s))\leq d(\psi, \phi(0))+d_C(\phi(0), \phi(s)),
\]
where $d_C$ denotes the length along the curve $C$. In particular,
we have the following triangle inequality
\[
d(\psi, \phi(1))\leq d(\psi, \phi(0))+d(\phi(0), \phi(1)).
\]
\end{theo}
\begin{proof}For any $\epsilon>0$, we can get a two parameter
families of smooth curve $C(t, s, \epsilon): \Phi(t, s,
\epsilon)\in \cH$ solving
\[
Q(D^2\Phi)=\epsilon
\] with the boundary conditions corresponding $(\psi, \phi(s)).$
Denote
\[
L(s, \epsilon)=d_{\epsilon}(\psi,
\phi(s))=\int_0^1dt\sqrt{\int_M\left|\frac{\p \Phi}{\p
t}\right|^2(1+\triangle \Phi)dg}
\]
and
\[
l(s)=d_C(\phi(0), \phi(s))=\int_0^sd\tau\sqrt{\int_M\left|\frac{\p
\Phi}{\p \tau}\right|^2(1+\triangle \Phi)dg}.
\]
Let $F(s, \epsilon)=L(s, \epsilon)+l(s)$. Taking derivatives,
\begin{eqnarray*}
\frac{d L(s, \epsilon)}{ds}&=&\int_0^1\frac{dt}{\sqrt{E(t, s,
\epsilon)}}\int_M\left(\frac{\p \Phi}{\p t}\frac{\p^2\Phi}{\p t\p
s}+\frac{1}{2}\Phi_t^2\triangle\Phi\Phi_s\right)dg\\
&=&\int_0^1\frac{dt}{\sqrt{E(t, s, \epsilon)}}\left\{\frac{\p}{\p
t}\int_M\Phi_t\Phi_s(1+\triangle\Phi)dg-\epsilon\int_M\Phi_sdg\right\}\\
&=&\left[\frac{1}{\sqrt{E}}\int_M\Phi_t\Phi_s(1+\triangle\Phi)dg\right]_0^1-\epsilon\int_0^1\frac{dt}{\sqrt{E(t,
s, \epsilon)}}\int_M\Phi_sdg\\
&&+\int_0^1dt\left(E(t, s,
\epsilon)^{-\frac{3}{2}}\int_M\Phi_t\Phi_s(1+\triangle
\Phi)dg\int_M\Phi_t\epsilon dg\right)\\
&\geq&\left[\frac{1}{\sqrt{E}}\int_M\Phi_t\Phi_s(1+\triangle\Phi)dg\right]_0^1-C\epsilon\\
&=&\frac{1}{\sqrt{E(1, s, \epsilon)}}\int_M\Phi_t(1, s,
\epsilon)\Phi_s(1, s, \epsilon)(1+\triangle \phi(s))dg-C\epsilon,
\end{eqnarray*}
where we use the fact $\Phi_s(0, s, \epsilon)=0$. Also we have
\[
\frac{d l(s)}{ds}=\sqrt{\int_M\Phi_s^2(1+\triangle \phi(s))dg}.
\]
By the Schwartz inequality, we have
\[
\sqrt{E(1, s, \epsilon)}\sqrt{\int_M\Phi_s^2(1+\triangle
\phi(s))dg}\geq\left|\int_M\Phi_t(1, s, \epsilon)\Phi_s(1, s,
\epsilon)(1+\triangle \phi(s))dg\right|,
\]
 it follows that
\[
\frac{dF(s, \epsilon)}{ds}\geq-C\epsilon.
\]
So $F(s, \epsilon)-F(0, \epsilon)\geq -C\epsilon$. Let
$\epsilon\rightarrow 0$, we get that
\[
d(\psi, \phi(1))\leq d(\psi, \phi(0))+d_C(\phi(0), \phi(1)).
\]
To get the triangle inequality, let $C$ be the curve $\Phi(s)$
solving
\[
Q(D^2\Phi)=\epsilon.
\]
And we can get that
\[d(\psi, \phi(1))\leq d(\psi, \phi(0))+d_\epsilon(\phi(0),
\phi(1)).
\]
Let $\epsilon\rightarrow 0$ again, we get
\[
d(\psi, \phi(1))\leq d(\psi, \phi(0))+d(\phi(0), \phi(1)).
\]
\end{proof}

\begin{rem}
Theorem \ref{T-7-8} is used  to show that the length of any simple curve is always longer than the 
length of the geodesic between two end points (see Corollary \ref{c-5-9}).  This will follow directly from Theorem \ref{T-7-8}
with the assumption that $\psi$ is  not on the curve $C$ and  Lemma \ref{L-5-7}. This 
lapse of the argument was  kindly pointed out  by the anonymous referee. Indeed, the previous version of the paper did not assume that $\psi$ is not on the curve. 
We thank the referee for pointing out this. 
However this does not affect all the main results in this section in view of the result in Lemma \ref{L-5-7}. 
\end{rem}

\begin{cor}\label{c-5-9}The geodesic distance between any two points in $\cH$
realizes the minimum of the lengths over all possible paths.
\end{cor}
\begin{proof} We just need to prove that the result holds for any smooth simple curve $C: \phi(s): [0, 1]\rightarrow
\cH$. Take $\psi\rightarrow \phi(0)$ in $C^4$  such that $\psi$ is not on the curve $C$.  By Theorem \ref{T-7-8}, we get that
\[
d(\psi, \phi(1))\leq d(\psi, \phi(0))+d_C(\phi(0), \phi(1)).
\]
Now let $\psi\rightarrow \phi(0)$ and by Lemma \ref{L-5-7}, we get
\[
d(\phi(0), \phi(1))\leq d_C(\phi(0), \phi(1)). 
\]
\end{proof}
\begin{theo}The space $(\cH, d)$ is a  metric space. Moreover, the
distance function is at least $C^1$.
\end{theo}
\begin{proof}The only thing we need to show is the
differentiability of the distance function. Follow from the proof
of Theorem \ref{T-7-8}, for any $\epsilon>0$ we have
\[
\left|\frac{dL(s, \epsilon)}{ds}-\frac{1}{\sqrt{E(1, s,
\epsilon)}}\int_M\Phi_t\Phi_s(1+\triangle \Phi)dg\right|\leq
C\epsilon.
\]
It easily follows  from above that
\[
\lim_{s\rightarrow s_0}\frac{d(\psi, \phi(s))-d(\psi,
\phi(s_0))}{s-s_0}=\frac{1}{\sqrt{E(1, s_0,
\epsilon)}}\int_M\Phi_t(1, s_0)\Phi_s(1, s_0)(1+\triangle
\phi(s_0))dg.
\]
\end{proof}
\subsection{The curvature of $\cH$}
Donaldson \cite{Donaldson2007} has shown that the space $\cH$ has
non-positive sectional curvature formally with the natural metric.
However, we can only demonstrate that the geodesic equation has a
weak $C^2$ solution. To overcome this difficulty, we show that the
space $\cH$ has non-positive curvature in the Alexandrov's sense
by following Calabi-Chen \cite{CC2002}.
\begin{theo}\label{T-7-11}The space $(\cH, d)$ is a non-positive curved space on any
Riemannian manifold $(X, g)$ in the following sense. Let $A, B, C$
be three points in $\cH$. For any $\lambda\in [0, 1]$, let $P$ be
the point on the geodesic path connecting $B$ and $C$ such that
$d(B, P)=\lambda d(B, C)$ and $d(P, C)=(1-\lambda)d(B, C)$. Then
the following inequality holds:
\[
d^2(A, P)\leq (1-\lambda)d^2(A, B)+\lambda d^2(A,
C)-\lambda(1-\lambda)d^2(B, C).
\]
\end{theo}
To prove this theorem, we need the following lemma, which in
essence says that the Jacobi vector field along any geodesic grows
super-linearly.
\begin{lemma}\label{L-7-12}Let $\Phi(t, s, \epsilon)$ be the two parameter
families of approximation of geodesics as in Lemma \ref{L-7-7}.
Let $Y=\Phi_s$ be the deformation vector fields and $X=\Phi_t$ be
the tangent vector fields along the approximating geodesic. Then
we have
\[
\langle D _XD _XY, Y\rangle\geq 0.
\]
Note $D$ is the covariant derivative defined in Section 2.
Moreover, if we assume that $Y(0, s, \epsilon)=0$, we have at
$t=1$
\[
\langle Y, D_XY\rangle\geq \langle Y, Y\rangle
\]
\end{lemma}
\begin{proof}By definition, the length of $Y$ is given by
\[
|Y|^2=\langle Y, Y\rangle=\int_M\Phi_s^2(1+\triangle \Phi)dg.
\]
It follows that
\[
\frac{1}{2}\frac{\p}{\p t}|Y|^2=\langle D_XY, Y\rangle=\langle
D_YX, Y\rangle.
\]
Let $K(X, Y)$ be the sectional curvature of the space $\cH$ at
point $\Phi(t, s, \epsilon).$ By the calculation of Donaldson
\cite{Donaldson2007}, we know that
\[
K(X, Y)\leq 0.
\]
We should emphasize that the calculation in \cite{Donaldson2007}
is algebraic and since $\Phi(t, s, \epsilon)\in \cH$, so we can
use in our setting. Therefore, we have
\begin{eqnarray*}
\frac{1}{2}\frac{\p^2}{\p t^2}|Y|^2&=&\langle D_XY,
D_YX\rangle+\langle D_XD_YX, Y\rangle\\
&=&|D_XY|^2+\langle D_YD_XX, Y\rangle-K(X, Y)\\
&\geq&|D_XY|^2+\langle D_YD_XX, Y\rangle.
\end{eqnarray*}
By definition, it is easy to see that
\[
D_XX=\Phi_{tt}-\frac{1}{1+\triangle \Phi}\langle \nabla \Phi_t,
\nabla \Phi_t\rangle=\frac{\epsilon}{1+\triangle \Phi}.
\]
Also we can get
\begin{eqnarray*}
D_YD_XX=D_Y\frac{\epsilon}{1+\triangle \Phi}&=&\frac{\p}{\p
s}\left(\frac{\epsilon}{1+\triangle
\Phi}\right)-\frac{1}{1+\triangle \Phi}\nabla \Phi_s\cdot \nabla
\left(\frac{\epsilon}{1+\triangle \Phi}\right)\\
&=&-\frac{\epsilon\triangle
\Phi_s}{(1+\triangle\Phi)^2}-\frac{1}{1+\triangle \Phi}\nabla
\Phi_s\cdot \nabla \left(\frac{\epsilon}{1+\triangle \Phi}\right).
\end{eqnarray*}
It follows that
\begin{eqnarray*}
\langle D_YD_XX,
Y\rangle&=&\int_M\Phi_s\left\{-\frac{\epsilon\triangle
\Phi_s}{(1+\triangle\Phi)^2}-\frac{1}{1+\triangle \Phi}\nabla
\Phi_s\cdot \nabla \left(\frac{\epsilon}{1+\triangle
\Phi}\right)\right\}(1+\triangle \Phi)dg\\
&=&\epsilon\int_M\left\{-\triangle
\Phi_s\frac{\Phi_s}{1+\triangle\Phi}-\Phi_s\nabla \Phi_s\nabla
\left(\frac{1}{1+\triangle\Phi}\right)\right\}dg\\
&=&\epsilon\int_M\frac{|\nabla
\Phi_s|^2}{1+\triangle\Phi}dg=\epsilon
\left\langle\frac{\nabla\Phi_s}{1+\triangle\Phi},
\frac{\nabla\Phi_s}{1+\triangle\Phi}\right\rangle\geq 0.
\end{eqnarray*}
In particular, we have
\[
\langle D_XD_XY, Y\rangle=\langle D_YD_XX, Y\rangle-K(X, Y)\geq 0.
\]
It follows that \[\frac{1}{2}\frac{\p^2}{\p t^2}|Y|^2\geq|D_XY|^2.
\]
Note that
\[
\frac{1}{2}\frac{\p^2}{\p t^2}|Y|^2=|Y|\frac{\p ^2}{\p
t^2}|Y|+\left(\frac{\p }{\p t}|Y|\right)^2,
\]
and it is easy to see that \[|D_XY|^2\geq \left(\frac{\p }{\p
t}|Y|\right)^2,\] we get that
\[
\frac{\p^2 }{\p t^2} |Y|\geq 0.
\]
Namely, $|Y|$ is a convex function of $t$. If $Y(0)=0$, it follows
that
\[
\frac{\p}{\p t}|Y(t)|_{t=1}\geq |Y(1)|,
\]
namely at $t=1$, \[ \langle D_XY, Y\rangle\geq \langle Y,
Y\rangle.
\]
\end{proof}
Now we are in the position to prove Theorem \ref{T-7-11}.
\begin{proof} For any $A, B, C\in \cH$, take $\phi_0(s)\equiv A$,
$\phi_1(s)$ to be an approximation of geodesic connecting $B, C$,
then apply Lemma \ref{L-7-7}, we can get a two parameter families
$\Phi(t, s, \epsilon)\in \cH$. Denote $P(s)$ to be the point
$\Phi(1, s, \epsilon)$. Let $E(s)$ be the energy of curve $\Phi(t,
s, \epsilon)$ connection $A$ and $P(s)$. When $\epsilon
\rightarrow 0,$ it is easy to see that $E(s)$ is the square of the
geodesic distance from $A$ to $P(s)$. We have
\[
E(s)=\int_0^1\int_M\Phi_t^2(1+\triangle \Phi)dt.
\]
As in Theorem \ref{T-7-8}, it is easy to get that
\[
\frac{1}{2}\frac{dE}{ds}=\int_M\Phi_t\Phi_s(1+\triangle\Phi)dg|_{t=1}-\epsilon\int_0^1\int_M\Phi_sdg.
\]
Use the notation in Lemma \ref{L-7-12}, we get
\[
\frac{1}{2}\frac{dE}{ds}=\langle X,
Y\rangle_{t=1}-\epsilon\int_0^1\int_M\Phi_sdg.
\]
Then it follows that
\begin{eqnarray*}\frac{1}{2}\frac{d^2E}{ds^2}&=&\frac{d}{ds}\langle X, Y\rangle_{t=1}-\epsilon
\int_0^1\int_M\Phi_{ss}dgdt\\
&=&\langle D_YX, Y\rangle_{t=1}+\langle X, D_YY\rangle_{t=1}
-\epsilon
\int_0^1\int_M\Phi_{ss}dgdt\\
&\geq&\langle Y, Y\rangle_{t=1}+\langle X,
D_YY\rangle_{t=1}-C\epsilon,
\end{eqnarray*}
where we use Lemma \ref{L-7-12}. Note at $t=1$,
\[
D_YY=\frac{\epsilon}{1+\triangle \Phi}
\]
We have
\[
\langle X, D_YY\rangle_{t=1}=\epsilon \int_M\Phi_tdg.
\]
Also we have
\[
\langle Y, Y\rangle=\int_M\Phi_s^2(1+\triangle \Phi)dg\geq
E(\Phi(1, \cdot, \epsilon))-C\epsilon,
\]
where $E(\Phi(1, \cdot, \epsilon))$ is the energy of the path
$\Phi(1, s, \epsilon)$. We get that
\[
\frac{1}{2}\frac{d^2E}{ds^2}\geq E(\Phi(1, \cdot,
\epsilon))-C\epsilon.
\]
It follows that
\[
E(s)\leq (1-s)E(0)+sE(1)-s(1-s)E(\Phi(1, \cdot,
\epsilon))-C\epsilon.
\]
Now fix $s$, let $\epsilon \rightarrow 0$. Each energy element
approaches the square of the length of the path. Thus we get
\[
|AP(s)|^2\leq (1-s)|AB|^2+s|AC|^2-s(1-s)|BC|^2.
\]
\end{proof}

\noindent Xiuxiong Chen\\
xxchen@math.wisc.edu\\
Department of Mathematics\\
University of Wisconsin-Madison\\

\noindent Weiyong He\\
whe@uoregon.edu\\
Department of Mathematics\\
University of Oregon


\begin{thebibliography}{s2}
\bibitem{Arezzotian2003}C. Arezzo; G. Tian, {\it Infinite geodesic rays in the space of
K\"ahler potentials.}  Ann. Sc. Norm. Super. Pisa Cl. Sci. (5) 2
(2003),  no. 4, 617--630.
\bibitem{Blocki}Z. Blocki, {\it On geodesics in the space of K\"ahler metrics,} preprint. 
\bibitem{Caca}L. A. Caffarelli,  X. Cabr\'e,  {\it Fully nonlinear elliptic equations,} American Mathematical Society Colloquium Publications, 43. American Mathematical Society, Providence, RI, 1995.
\bibitem{CC2002}E. Calabi, X. X. Chen, {\it The space of K\"ahler metrics
II},  J. Differential. Geom. {\bf 61} (2002), no.2, 173-193.
\bibitem{Chen2000}X. X. Chen, {\it The space of K\"ahler metrics},
J. Differential. Geom. {\bf 56} (2000), no.2, 189-234.
\bibitem{Chen2005} X. X. Chen, {\it Space of K\"ahler metrics III--On the lower bound of the
Calabi energy and geodesic distance}, arXiv:math/0606228.
\bibitem{ChenWu}Y. Chen, L. Wu; {\it Second order elliptic equations and systems of elliptic equations}, Translations of Mathematical Monographs, Amer. Math. Soc., 1998.  
\bibitem{Donaldson1997}S. Donaldson, {\it Symmetric spaces, K\"ahler geometry and Hamiltonian
dynamics},  Northern California Symplectic Geometry Seminar,
13--33, Amer. Math. Soc. Transl. Ser. 2, 196, Amer. Math. Soc.,
Providence, RI, 1999.
\bibitem{Donaldson2002}S. K. Donaldson, {\it Scalar curvature and stability of toric
varieties}, J. Differential Geom. {\bf 62} (2002),  no. 2,
289--349.
\bibitem{Donaldson2005}S. K. Donaldson, {\it Lower bound of the Calabi
functional},  J. Differential Geom.  70  (2005),  no. 3, 453--472.
\bibitem{Donaldson2007}S. Donaldson, {\it Nahm's equations and free-boundary
problems}, arXiv:0709.0184.
\bibitem{Evans1982}L. Evans, {\it Classical solutions of fully nonlinear, convex, second order elliptic
equations,} Acta Math. 148 (1982), 47--157.
\bibitem{Giltru1998}D. Gilbarg, N. Trudinger, {\it Elliptic partial differential equations of second order.}
\bibitem{Guan} B. Guan, {\it The Dirichlet problem for complex Monge-Amp\`ere equations and regularity of the pluri-complex Green function}.  Comm. Anal. Geom. 6 (1998), no. 4, 687--703.
\bibitem{Guan-Spruck} B. Guan, J. Spruck, {\it Boundary-value problems on $S^n$ for surfaces of constant Gauss curvature}. Ann. of Math. (2) 138 (1993), no. 3, 601--624. 
\bibitem{Guan2} P . F. Guan, {\it $C^2$ a priori estimates for degenerate Monge-Amp\`ere equations.}   Duke Math. J. 86 (1997), no. 2, 323--346. 
\bibitem{He09}W.Y. He, {The Donaldson equation}, http://arxiv.org/abs/0810.4123.
\bibitem{Krylov1982}N. Krylov, {\it Boundedly inhomogeneous elliptic and parabolic equations,} (Russian)
Izv. Akad. Nauk SSSR Ser. Mat.  46  (1982), no. 3, 487--523.
English Translation in Math. USSR Izv. {\bf 20} (1983).
\bibitem{Krylov1982a}N. Krylov, {\it Boundedly inhomogeneous elliptic and parabolic equations in a domain,} (Russian)
Izv. Akad. Nauk SSSR Ser. Mat.  47  (1983),  no. 1, 75--108.
English Translation in Math. USSR Izv. {\bf 22}, 67-97 (1984).
\bibitem{Mabuchi87}T. Mabuchi, {\it Some symplectic geometry on compact K\"ahler manifolds,} I.  Osaka J. Math.  24  (1987),  
no. 2, 227--252.
\bibitem{Semmes96}S. Semmes, {\it Complex Monge-Amp\`ere and symplectic manifolds,} Amer. J. Math.  114  (1992),  no. 3, 495--550.
\bibitem{Yuan1} Y. Yuan, {\it A priori estimates for solutions of fully nonlinear special Lagrangian equations,} Ann. Inst. H. Poincare Anal. Non Linaire 18 (2001), 261-270.
\bibitem{Yuan2} Y. Yuan, {\it A Bernstein problem for special Lagrangian equations.}  Invent. Math. 150 (2002), no. 1, 117--125.
\end{thebibliography}
\end{document}